\documentclass[a4paper,12pt]{article}

\usepackage{amsthm,amsmath,stmaryrd,bbm,hyperref,geometry,color,authblk,bm}
\usepackage[utf8]{inputenc}
\usepackage{amssymb}
\usepackage[english]{babel}
\usepackage{graphicx}
\usepackage{amsfonts,amssymb}
\usepackage{verbatim}
\usepackage{enumitem}

\newcommand{\po}{\left(}
\newcommand{\pf}{\right)}
\newcommand{\co}{\left[}
\newcommand{\cf}{\right]}
\newcommand{\cco}{\llbracket}
\newcommand{\ccf}{\rrbracket}
\newcommand{\R}{\mathbb R}

\newcommand{\N}{\mathbb N} 
\newcommand{\dd}{\text{d}}
\newcommand{\na}{\nabla}
\newcommand{\1}{\mathbbm{1}} 
\newcommand{\bx}{\bm{x}}
\newcommand{\bv}{\bm{v}}
\newcommand{\by}{\bm{y}}
\newcommand{\bX}{\bm{X}}
\newcommand{\bY}{\bm{Y}}

\newcommand{\Var}{\mathrm{Var}}

\newtheorem{thm}{Theorem}
\newtheorem{assu}{Assumption}
\newtheorem*{assu*}{Assumption}
\newtheorem{lem}[thm]{Lemma}

\newtheorem{cor}[thm]{Corollary}

\newtheorem{rem}{Remark}

\title{Uniform log-Sobolev inequalities for mean field
particles   beyond flat-convexity}
\author{Pierre Monmarché}

\begin{document}
\maketitle

\begin{abstract}
In the nice recent work \cite{Songbo}, S. Wang  established uniform log-Sobolev inequalities for mean field particles when the energy is flat convex. In this  note we comment how to extend his proof to some semi-convex energies provided the curvature lower-bound is not too negative. It is not clear that this could be obtained simply by applying a posteriori a perturbation argument to Wang's result. Rather, we follow his proof and, at steps where the convexity is used, we notice that the uniform conditional or local functional inequalities assumed give some room to allow for a bit of concavity. In particular, this allows to recover other  previous results on non-convex systems at high temperature or weak coupling, and to consider situations which mix flat-convexity and weak coupling.
\end{abstract}

\section{Settings and results}\label{sec:intro}

\subsection{Notations}

Denote by $\mathcal P_2(\R^d)$ the set of probability measures on $\R^d$ with finite moments and by $\mathcal W_2$ the associated $L^2$ Wasserstein distance. Consider a function $F:\mathcal P_2(\R^d) \rightarrow \R$ and $N\in\N^*$. We write
\[\mu_{\bx} = \frac1N \sum_{i=1}^N \delta_{x_i} \]
the empirical measure of a vector $\bx=(x_1,\dots,x_N) \in (\R^d)^N$, and
\begin{equation}
\label{eq:UN}
U_N(\bx) = N F(\mu_{\bx})\,. 
\end{equation}
For $\bx\in\R^{dN}$ and $i\in\cco 1,N\ccf$, we write $\bx_{\neq i} = (x_j)_{j\in\cco 1,N\ccf\setminus\{i\}}$.

Assuming that $Z_N = \int_{\R^{dN}} e^{-U_N(\bx)} \dd \bx < \infty$, we are interested in the Gibbs measure 
\begin{equation}
\label{eq:Gibbs}
m_*^N(\bx) = \frac{e^{-U_N(\bx)}}{Z_N}\,. 
\end{equation}
Our goal is to prove, under suitable conditions on $F$, Poincaré or log-Sobolev inequalities for $m_N$, uniformly in $N$ large enough. Here we say that a probability measure $\mu$ on $\R^d$ satisfies a Poincaré inequality with constant $\rho>0$ if for all $f\in \mathcal C_b^1$ ($\mathcal C^1$ with bounded derivatives),
\[ \int_{\R^d} f^2 \dd \mu - \po \int_{\R^d} f\dd \mu \pf^2 \leqslant \frac{1}{\rho } \int_{\R^d}|\na f|^2 \dd \mu \,,\]
and we say that it satisfies a log-Sobolev inequality (LSI) with constant $\rho>0$ if for all probability measure $\nu \ll \mu$ with $\frac{\dd \nu}{\dd \mu} \in \mathcal C_b^1$,
\[ \mathcal H\po \nu|\mu\pf \leqslant \frac{1}{2\rho }\mathcal I(\nu|\mu)\,,\]
where the relative entropy and Fisher information are  respectively given by
\[\mathcal H(\nu|\mu) = \int_{\R^d} \ln\po \frac{\dd \nu}{\dd \mu}\pf \dd \nu\,,\qquad \mathcal I \po \nu|\mu\pf = \int_{\R^d} \left|\na \ln \frac{\dd \nu}{\dd \mu}\right|^2 \dd \nu\,.\]

Denoting $H(\nu) = \int_{\R^d} \nu \ln \nu$ the entropy of a measure $\nu\in\mathcal P_2(\R^d)$ (taken as $+\infty$ if $\nu$ doesn't have a Lebesgue density), we assume that the so-called free energy 
\[\mathcal F(\nu) = F(\nu)+H(\nu)\]
 is bounded from below and admits a global minimizer $m_\infty$ (which can typically be checked in practical cases using the La Salle invariance principle \cite{carrillo2023invariance}).

\subsection{Motivations and related works}\label{sec:motivations}

\paragraph{General bibliography.} The Gibbs measure~\eqref{eq:Gibbs} is the invariant measure of the Langevin process with potential $U_N$, which is the Markov diffusion $\bX_t$ on $\R^{dN}$ solving
\begin{equation}
\label{eq:Langevin}
\dd \bX_t = - \na U_N(\bX_t)\dd t + \sqrt{2} \dd \bm{B}_t\,,
\end{equation}
with $\bm{B}$ a $dN$-dimensional Brownian motion. The Poincaré and log-Sobolev inequalities give the convergence rate of this process toward equilibrium, respectively in the $\chi^2$ sense or in relative entropy (see e.g. Theorems~4.2.5 and 5.2.1 of \cite{BakryGentilLedoux}).  Among other things, they are also related to concentration properties of $m_*^N$, transport inequalities and estimates on the Poisson equation associated to~\eqref{eq:Langevin}; see \cite{BakryGentilLedoux,ane2000inegalites} for general considerations on these two inequalities.

 When $U_N$ is given by~\eqref{eq:UN}, the process \eqref{eq:Langevin} is a system of mean-field interacting particles. Under suitable conditions on $F$, propagation of chaos occurs, meaning that the empirical measure $\mu_{\bX_t}$ converges as $N\rightarrow \infty$ to the deterministic solution of a semi-linear PDE,  which is the Wasserstein gradient descent of the  free energy $\mathcal F$. We refer to \cite{ambrosio2005gradient} for details on Wasserstein gradient flows. As shown in \cite{Pavliotis}, a uniform-in-$N$ log-Sobolev inequality for $m^N_*$ implies a free energy/dissipation inequality (called non-linear log-Sobolev inequality in \cite{MonmarcheReygner})  for the non-linear limit flow which, with the vocabulary of gradient descent and optimization, corresponds to a Polyak-{\L}ojasiewicz inequality (it also implies non-linear versions of Talagrand transport inequalities). It implies that all the critical points of the free energy $\mathcal F$ are global minimizers, and a global convergence of this free energy to its global minimum along the flow. However in general the free energy may have other critical points than global minimizers \cite{Pavliotis,Tugautdoublewell,MonmarcheReygner}, in which case  the log-Sobolev constant of $m_*^N$ vanishes as $N\rightarrow \infty$. Then, moreover, the particle system is metastable and quantitative estimates on propation of chaos cannot be uniform in time, while they are under a uniform-in-$N$ log-Sobolev inequality for $m_*^N$, see \cite{Malrieu,M11,M33}. These questions, together with the local stability of  the non-linear flow, has recently been an active topic of research \cite{cormier2022stability,MonmarcheReygner,Tugautmonotone,Bashiri,bashiri2021metastability,carrillo2020long} (see references within for previous works).

Moreover, using modified entropies, a uniform log-Sobolev inequality for $m_*^N$ also gives uniform in $N$ long-time convergence rates and uniform in time propagation of chaos estimates for the kinetic Langevin process with energy $U_N$ and its mean-field limit, the Vlasov-Fokker-Planck equation \cite{M15,M11,MonmarcheReygner,Assadeck,Songbo2}, which is an important model of statistical physics and is also of interest for sampling, and also for the associated numerical schemes and non-linear Hamiltonian Monte Carlo \cite{bou2023nonlinear,M44,M47}.

Many tools are available for establishing functional inequalities such as Poincaré and log-Sobolev. However, apart in the independent case where these inequalities tensorize (see \cite{BakryGentilLedoux}) and in the log-concave case (thanks to the Bakry-Emery criterion \cite{BakryGentilLedoux} in the strongly-log-concave case, and  more generally for the Poincaré inequality \cite{cattiaux2020poincare}, in relation to the KLS conjecture \cite{chen2021almost,klartag2022bourgain,klartag2023logarithmic}), obtaining bounds which are independent from the dimension is usually challenging. For mean-field interacting systems, perturbation arguments \cite{AIDA1994448,holley1986logarithmic,CattiauxGuillin} or Lyapunov criteria \cite{cattiaux2009lyapunov,10.1214/11-AIHP447,10.1214/ECP.v13-1352,CATTIAUX20172361,BAKRY2008727} behave badly with $N$. 

For more details and general discussions on Poincaré and log-Sobolev inequalities in mean-field systems, we refer to \cite{Pavliotis,guillin2022uniform} and references within.

\paragraph{Specific recent works.} Particularly relevant to the present note, recently, motivated by high-dimensional statistics and machine learning, several works have focused on the case where $F$ is flat-convex, i.e. convex along linear interpolations:
\[F(t\mu + (1-t)\nu) \leqslant t F(\mu) + (1-t) F(\nu)\qquad \forall t\in[0,1],\ \mu,\nu\in\mathcal P_2(\R^d)\,. \]
It should be emphasized that this is completely different than saying that $U_N$ is convex (i.e. $m_*^N$ log-concave). For instance, if $F(\mu) = \int_{\R^d} V \dd \mu$ for some $V:\R^d\rightarrow \R$, then $\mu\mapsto F(\mu)$ is linear and thus always flat-convex, although $U_N(\bx) = \sum_{i=1}^N V(x_i)$ is convex if and only if $V$ is. In the flat-convex case, under suitable additional assumptions, a non-linear log-Sobolev inequality for the mean-field non-linear flow has been established in \cite{chizat,nitanda2022convex}, and an approximate convergence of the particle system toward the non-linear stationary solution is given in \cite{ChenRenWang,nitanda2024improved}.

The  very recent works~\cite{Songbo,Chewietal,kook2024sampling} prove a uniform-in-$N$ log-Sobolev inequality for $m_*^N$ in the flat-convex case (under suitable additional conditions), with two very different approaches. The works~\cite{Chewietal,kook2024sampling} are based on the reverse heat flow method of Kim and Milman~\cite{kim2012generalization}. They show that $m_*^N$ is the image of the standard Gaussian distribution on $\R^{dN}$ by a Lipschitz map (with a Lipschitz constant independent from $N$), which is stronger than the log-Sobolev inequality. On the other hand, this approach requires further assumptions than \cite{Songbo} and, related to our motivation in the present note, with this method,  the flat-convexity assumption may possibly be a rigid constraint. 

The rest of the present work builds upon~\cite{Songbo}. 

\subsection{Results}

To establish a Poincaré inequality, we work under the following set of conditions.

\begin{assu}\label{assu:Poincare}\

\begin{enumerate}[label=(\roman*)]
\item The partition function $Z_N = \int_{\R^{dN}} e^{-U_N}$ is finite, the free energy $\mathcal F$ is bounded and admits a minimizer $m_\infty$.\label{assu1i}
\item $F$ is $\lambda$-semi-convex for some $\lambda\geqslant 0$, in the sense that for all $\mu,\nu\in\mathcal P_2(\R^d)$ and $t\in[0,1]$,
\begin{equation}
\label{eq:courbure}
F(t\mu + (1-t)\nu) \leqslant t F(\mu) + (1-t) F(\nu) +  t(1-t) \frac{\lambda}2\mathcal W_2^2 (\nu,\mu)\,.
\end{equation}\label{assu1ii}
\item The flat derivatives  (as defined in \cite{carmona2018probabilistic})
\[\frac{\delta F}{\delta m}: \mathcal P_2(\R^d) \times \R^d \rightarrow \R\,,\qquad  \frac{\delta^2 F}{\delta^2 m} :\mathcal P_2(\R^d) \times \R^d \times \R^d \rightarrow \R\]
exist, are jointly continuous, and are $\mathcal C^2$ in the space variables. We denote
 \[D_m F(m,x) = \na_x \frac{\delta F}{\delta m}(m,x)\,,\qquad  D_m^2 F(m,x,x') = \na_{x,x'}^2 \frac{\delta^2 F}{\delta^2 m}(m,x,x')\,.\]\label{assu1iii}
 \item There exists $M_{mm}^F \geqslant 0$ such that for all $m\in \mathcal P_2(\R^d),x,x' \in\R^d$, 
 $|D_m^2 F(m,x,x')|\leqslant M_{mm}^F$ (where $|\cdot|$ stands for the operator norm with respect to the Euclidean norm).  \label{assu1iv}
 \item For $N\in\N^*$, there exists $\rho_N>0$ such that for all $\by \in\R^{d(N-1)}$, the conditional density of $X_1$ given $\bX_{\neq 1} =\by$ when $\bX \sim m_*^N$ (i.e. the density proportional to $x_1\mapsto e^{-N F(\mu_{\bx})}$ at fixed $\bx_{\neq 1}=\by$) satisfies a Poincaré inequality with constant $\rho_N$.  \label{assu1v}

\end{enumerate}
\end{assu}

We are interested in cases where, in item $\ref{assu1v}$, the conditional Poincaré inequality is in fact uniform in $N$, meaning that $\liminf \rho_N >0$ as $N\rightarrow \infty$.  Establishing this is often much simpler than obtaining a uniform-in-$N$ inequality for $m_*^N$ (the former may hold while the latter fails, as discussed e.g. in~\cite{MonmarcheReygner}), thanks to all the tools mentioned in  Section~\ref{sec:motivations} (e.g. perturbation results, Lyapunov conditions), since the conditional distribution of $X_1$ given $\bX_{\neq 1}$ is a probability measure on $\R^d$ (i.e. $N$ does not intervene in the dimension). See also Remark~\ref{eq:SongboAssume} below.

When $\inf \rho_N>0$, the conditions in Assumption~\ref{assu:Poincare} are the same as in S. Wang's work \cite{Songbo} for proving the Poincaré inequality, except that in the latter~\eqref{eq:courbure} holds with $\lambda=0$.

The following result is in the spirit of \cite[Theorem 1]{guillin2022uniform} (which is only concerned with pairwise interactions, corresponding to $F(\mu) = \int_{\R^d} V \dd \mu + \int_{\R^{2d}} W \dd \mu^{\otimes 2}$ for some potentials $V(x),W(x,y)$).

\begin{thm}\label{thm:Poincare}
Under Assumption~\ref{assu:Poincare}, $m_*^N$ satisfies a Poincaré inequality with constant $  \rho_N - \lambda - \frac{M_{mm}^F}{N}$, provided this is positive.
\end{thm}

The proof is given in Section~\ref{sec:Poincare}. Under Assumption~\ref{assu:Poincare}, the condition $\liminf \rho_N > \lambda$ to get a uniform in $N$ Poincaré inequality is sharp, as we discuss in Section~\ref{sec:sharpness}, and so is the Poincaré constant.

For the log-Sobolev inequality, we add the following conditions.

\begin{assu}\label{assu:LSI}\

\begin{enumerate}[label=(\roman*)]
\item There exists $\mathcal C:\mathcal P_2(\R^d)\times\mathcal P_2(\R^d) \rightarrow \R_+$ such that for all $\mu,\nu\in\mathcal P_2(\R^d)$ and $t\in[0,1]$,
\begin{equation}
\label{eq:courbure-cost}
F(t\mu + (1-t)\nu) \leqslant t F(\mu) + (1-t) F(\nu) +  t(1-t) \mathcal C (\nu,\mu)\,.
\end{equation}
Moreover, there exists $\lambda'\geqslant 0$ such that for all $N \in\N$ there exists $\alpha_N >0$ such that for all $m^N \in\mathcal P_2(\R^{dN})$,
\begin{equation}
\label{eq:cost-bound}
\int_{\R^{dN}} \mathcal C(\mu_{\bx},m_\infty) m^N(\dd \bx) \leqslant \frac{\lambda'}N \mathcal W_2^2 \po m^N,m_\infty^{\otimes N}\pf + \frac{\alpha_N}N\,.
\end{equation}
\label{assu2i}
\item There exists $\rho>0$ such that for all $m\in\mathcal P_2(\R^d)$, the probability measure   with density proportional to $\exp(-\frac{\delta F}{\delta_m}(m,\cdot))$ satisfies a LSI with constant $\rho$. \label{assu2ii}
\end{enumerate}

\end{assu} 

Of course, \eqref{eq:courbure} implies \eqref{eq:courbure-cost} with $\mathcal C = \frac{\lambda}{2}\mathcal W_2^2$, and then an inequality of the form~\eqref{eq:cost-bound} can be obtained using~\cite{FournierGuillin}. However, we will discuss in Section~\ref{sec:W2C} why  this choice cannot give uniform-in-$N$ at the end, making it necessary to work with other choices of $\mathcal C$. When~\eqref{eq:courbure-cost} holds with $\mathcal C=0$, then \eqref{eq:cost-bound} is trivial and thus we are back to the conditions considered in~\cite{Songbo}. We also refer to \cite{MonmarcheReygner} where~\eqref{eq:courbure-cost} for suitable $\mathcal C$ is used to establish (possibly local) non-linear log-Sobolev inequalities for the non-linear limit flow associated to $\mathcal F$ (see in particular \cite[Section 2.3]{MonmarcheReygner}  where it is proven that the semi-convexity condition~\eqref{eq:courbure} together with the uniform LSI of Assumption~\ref{assu:LSI}$\ref{assu2ii}$ with $4\lambda<\rho$ implies a global non-linear LSI).

\begin{rem}\label{eq:SongboAssume} Two comments about our assumptions:
\begin{itemize}
\item As discussed in \cite[Section 4]{Songbo}, under the additional condition that $\na_x \frac{\delta^2 F}{\delta^2 m}$ is bounded,  Assumption~\ref{assu:LSI}$\ref{assu2ii}$ implies that Assumptions~\ref{assu:Poincare}$\ref{assu1v}$  holds with $\rho_N \rightarrow \rho$ as $N\rightarrow \infty$.
\item Following~\cite{Songbo}, we will rely on a result from~\cite{ChenRenWang}. In the latter, it is assumed that $\mu \mapsto D_m F(\mu,y)$ is $M_{mm}^F$-Lipschitz with respect to the $\mathcal W_1$ Wasserstein distance for all $y\in\R^d$. This is equivalent to the bound on the second derivative, since
\begin{align*}
|D_m F(\mu,y) - D_m F(\nu,y)| &= \left|\int_0^1 \int_{\R^d} \na_y \frac{\delta^2 F}{\delta^2 m}\po (1-s)\mu + s\nu,y,y'\pf(\mu-\nu)(\dd y')\dd s\right|\\
& \leqslant   \int_0^1 \left\|  \na_{y,y'} \frac{\delta^2 F}{\delta^2 m}\po (1-s)\mu + s\nu,\cdot,\cdot\pf \right\|_{\infty} \mathcal W_1(\mu,\nu)  \dd s \\
& \leqslant M_{mm}^F \mathcal W_1(\mu,\nu)\,.
\end{align*}
The reverse implication is obtained by writing the Lipschitz condition with $\nu$ taken as the image of $\mu$ by an arbitrary small translation.
\end{itemize}

\end{rem}

For $\rho,\delta>0$, we say that a probability measure $\mu$ on $\R^d$ satisfies a defective LSI with constants $(\rho,\delta)$ if for all probability measure $\nu \ll \mu$ with $\frac{\dd \nu}{\dd \mu} \in \mathcal C_b^1$,
\[ 2\rho \mathcal H\po \nu|\mu\pf \leqslant  \mathcal I(\nu|\mu) + \delta\,.\]

A defective LSI together with a Poincaré inequality implies a usual (tight) LSI, see \cite[Lemma 5.1.4]{BakryGentilLedoux}. We state the defective inequality for $m_*^N$ separately as it is also of interest per se, as we discuss in Section~\ref{sec:exemple-defective}

\begin{thm}\label{thm:LSI}
Let $\varepsilon\in(0,1)$. Under Assumptions~\ref{assu:Poincare} and \ref{assu:LSI},  assume furthermore that $ 4\lambda'< \rho$.   Introduce the following notations:
\begin{align}
N_0 &= \frac{4 M_{mm}^F}{\rho - 4  \lambda'} \po 4 +  \frac{3M_{mm}^F (\varepsilon^{-1}-1)}{2\rho(1-\varepsilon)} \pf \nonumber\\ 
\tilde\lambda_N &=  \lambda' + \frac{M_{mm}^F}N \po 4 +  \frac{3M_{mm}^F (\varepsilon^{-1}-1)}{2\rho(1-\varepsilon)} \pf  \label{eq:lambdatilde}\\
\delta_N &= 4\rho(1-\varepsilon)\po 2 \alpha_N  + \frac{M_{mm}^F d}{\rho } \po \frac52 +  \frac{3M_{mm}^F (\varepsilon^{-1}-1)}{4\rho(1-\varepsilon)} \pf \pf \nonumber \\
\beta_N &=\frac{2\tilde \lambda_N}{\rho}\po  1 - \frac{2\tilde \lambda_N}{\rho}\pf ^{-1} \label{eq:betaN}\,. 
\end{align}
 Then, for all $N>N_0$, $\beta_N\in(0,1)$ and  $m_*^N$ satisfies a defective LSI with constants $(\rho_{N,*}',\delta_N)$ where
\[\rho_{N,*}' =   2 (1-\varepsilon) (1-\beta_N)  \rho \,.  \]
\end{thm}

This result is proven in Section~\ref{sec:LSI}. Combining it with the Poincaré inequality of Theorem~\ref{thm:Poincare} gives an LSI (we use \cite[Proposition 5]{Songbo}  for the explicit constant in the following result).

\begin{cor}\label{cor:LSI}
Under the settings and with the notations of Theorem~\ref{thm:LSI}, for all $N>N_0$ such that moreover  $ \rho_N - \lambda - \frac{M_{mm}^F}{N}>0$,  $m_*^N$ satisfies a log-Sobolev inequality with constant
\[\rho_{N,*} =   \rho_{N,*}'\po 1+\frac{\delta_N}{4\po \rho_N - \lambda - \frac{M_{mm}^F}{N}\pf }\pf^{-1}   \,.  \]
In particular, assuming moreover that $\liminf\rho_N > \lambda$, there exists a constant $\eta >0$ such that  $\rho_{N,*}\geqslant \eta/(1+\alpha_N)$ for all $N$ large enough.
\end{cor}

In cases where $\alpha_N$ and $1/\rho_N$ are uniformly bounded in $N$, this gives a uniform-in-$N$ LSI inequality provided $\lambda,\lambda'$ are not too large. See Section~\ref{sec:examplequimarche} for such examples. In these cases, this result should be compared to \cite[Theorem 8]{guillin2022uniform}, which also establishes uniform LSI under conditions that allows for non flat convex energies. In this work, the analogue of our restrictions on $\lambda,\lambda'$  is Zegarlinski's condition, \cite[Equation~(26)]{guillin2022uniform}. First, the work of Guillin et al is restricted to pairwise interactions, i.e. $F(\mu) = \int_{\R^d} V \dd \mu + \int_{\R^{2d}} W \mu^{\otimes 2}$. More importantly, the condition \cite[Equation~(26)]{guillin2022uniform} imposes a uniform bound on the coupling between particles and is not able to take into account that some parts of the energy may be flat-convex, contrary to our result. In other words, we treat the flat-convex cases of \cite{Songbo} and the weakly coupled cases of \cite{guillin2022uniform} in a single framework, which enables to treat situations which mixes the two cases, as in Section~\ref{sec:particularexemple} below. It is not clear that these situations could be treated as a small perturbation of the convex case or as a convex perturbation of the weakly coupled case.

\subsection{Discussion and examples}\label{sec:exemple}

\subsubsection{Sharpness of the restriction on $\lambda$ for Poincaré}\label{sec:sharpness}

Theorem~\ref{thm:Poincare} yields a uniform-in-$N$ Poincaré inequality provided $\liminf_{N\rightarrow \infty} \rho_N > \lambda$. This condition  is sharp. Indeed, consider, as in \cite[Remark 3]{guillin2022uniform}, the case where $d=1$,
\begin{equation}
\label{eq:Fmuquadra}
F(\mu) = \frac12 \int_{\R}x^2 \mu(\dd x) -  \frac{a}2 \po \int_{\R} x \mu(\dd x)\pf ^2 \,,
\end{equation} 
for $a>0$. For $x,y\in\R$,
\[F\po (1-t)\delta_x+ t \delta_y\pf - (1-t) F(\delta_x) - t F(\delta_y) = \frac{a}2 t(1-t) |x-y|^2 = \frac{a}2 t(1-t)  \mathcal W_2^2(\delta_x,\delta_y)\,. \]
More generally, for any $\nu,\mu \in \mathcal P_2(\R)$, 
\[ F(t\mu + (1-t)\nu) -  t F(\mu) + (1-t) F(\nu)  = \frac{a}2 t(1-t) \po \int_{\R} x (\mu-\nu)(\dd x)\pf^2 \leqslant \frac{a}2t(1-t) \mathcal W_2^2(\nu,\mu)\,.\] 
In other words, \eqref{eq:courbure} holds with the optimal constant $\lambda=a$. Moreover, $\frac{\delta^2 F}{\delta^2 m}(m,x,x')=a xx'$ for all $m\in\mathcal P(\R)$, $x,x'\in\R$, from which $D_m^2 F=a$, which gives Assumption~\ref{assu:Poincare}$\ref{assu1iv}$. Finally, for $N>a$, for fixed $(x_j)_{j\in\cco 1,N\ccf}$, 
\[x_1 \mapsto e^{-N F(\mu_{\bx})} \propto \exp\po - \frac{1}{2} \po 1-\frac{a}{N}\pf  \po x_1 - \frac{a}{N-a}\sum_{j=2}^N x_j\pf^2   \pf\,. \]
This is the density of a Gaussian distribution with variance $\po 1-\frac{a}{N}\pf^{-1}$, whose optimal log-Sobolev constant is thus  $  1-\frac{a}{N}$. This means that Assumption~\ref{assu:Poincare}$\ref{assu1v}$ holds, with $\rho_N = 1 -\frac{a}{N}$ (which is optimal).

As a summary, in this case, with the optimal constants, $\rho_N - \lambda - \frac{M_{mm}^F}{N} = 1 -a \po 1+ \frac{2}{N}\pf $. On the other hand,
\[U_N(\bx) = \frac12 |\bx|^2 - \frac{a}{2N} \po \sum_{i=1}^N x_i\pf^2 \]
is a quadratic form which is not definite positive as soon as $a\geqslant 1$ (taking for instance $x_1=\dots=x_N$), which means that $m_*^N$ is not even defined in that case.  Moreover, when $a<1$, the variance of $m_*^N$ blows up as $a\rightarrow 1$, and thus its Poincaré constant goes to $0$.  More precisely, taking again $x_1=\dots=x_N$, we see that $1-a$ is an eigenvalue of $\na^2 U_N$, and thus the optimal Poincaré inequality of $m_N^*$ is smaller than $1-a$. In other words, in this simple case, Theorem~\ref{thm:Poincare} gives the optimal constant up to the vanishing term $2a/N$.

\subsubsection{Using only the semi-convexity}\label{sec:W2C}

 Assumption~\ref{assu:Poincare} implies \eqref{eq:courbure-cost} with $\mathcal C(\nu,\mu)= \frac{\lambda}{2}\mathcal W_2^2(\nu,\mu)$. Moreover, considering $(\bX,\bY)$ an optimal $\mathcal W_2$ coupling of $m^N$ and $m_\infty^{\otimes N}$, for any $\varepsilon>0$,
  \begin{align*}
 \mathbb E \po \mathcal W_2^2(\delta_{\bX},m_\infty)\pf &\leqslant (1+\varepsilon) \mathbb E \po \mathcal W_2^2(\delta_{\bX},\delta_{\bY})\pf + (1+\varepsilon^{-1})\mathbb E \po \mathcal W_2^2(\delta_{\bY},m_\infty)\pf \\
 &= (1+\varepsilon) \frac1N   \mathcal W_2^2\po m^N,m_\infty^{\otimes N}\pf + (1+\varepsilon^{-1})\mathbb E \po \mathcal W_2^2(\delta_{\bY},m_\infty)\pf\,.
  \end{align*}
Thanks to \cite[Theorem 1]{FournierGuillin}, the last term is of order $N^{-2/d}$ for $d>2$ (and this is optimal for measures with a density with respect to Lebesgue). In other words,  under   Assumption~\ref{assu:Poincare}, \eqref{eq:cost-bound} holds for $\mathcal C = \frac{\lambda}{2}\mathcal W_2^2$ with $\lambda'=(1+\varepsilon) \lambda$ that can be taken arbitrarily close to $\lambda$ and $\alpha_N = C_\varepsilon N^{1-2/d}$ for some $C_\varepsilon>0$. Applying Corollary~\ref{cor:LSI} then gives a log-Sobolev constant of order $N^{-1+2/d}$ (provided $4\lambda <\rho$ and $\liminf\rho_N >\lambda$). We cannot get better with our approach if we only use~\eqref{eq:courbure}, which is why in Assumption~\ref{assu:LSI} we allow some flexibility in the choice of $\mathcal C$. Indeed, as discussed in Section~\ref{sec:examplequimarche} below, there are cases of interest where Assumption~\ref{assu:LSI} holds with $\alpha_N$ uniformly bounded in $N$, so that Corollary~\ref{cor:LSI} gives a log-Sobolev constant uniform in $N$. 

That being said, when $\alpha_N$ is of order $N^{1-2/d}$, the defective LSI of Theorem~\ref{thm:LSI} may still be of interest, see Section~\ref{sec:exemple-defective}. 

\subsubsection{Uniform LSI for some parametrized energies}\label{sec:examplequimarche}

Assume that, for some Hilbert space $\mathbb H$ with norm $\|\cdot\|$,
\[F(\mu) = F_0(\mu) + R \po \int_{\R^d} \varphi \dd \mu\pf \,,\]
where $F_0$ is flat-convex, $\varphi:\R^d \rightarrow \mathbb H$ and $R:\mathbb H \rightarrow \R$. For examples and motivations for this form of energies, see \cite{MonmarcheReygner,chizat2018global} (besides, this is the case of~\eqref{eq:Fmuquadra} for instance). Assume that $R$ is $2\alpha$-semi-convex for some $\alpha\geqslant 0$, meaning that $R+\alpha\|\cdot\|^2$ is convex. Together with the flat-convexity of $F_0$, and since $\mu \mapsto \int_{\R^d} \varphi \dd \mu$ is linear, it is readily checked that this implies~\eqref{eq:courbure-cost} with
\[\mathcal C(\nu,\mu) = \alpha \left\|\int_{\R^d} \varphi \dd (\nu -\mu)\right\|^2\,. \]
Then, considering an optimal $\mathcal W_2$ coupling $\bX,\bY$ of $m^N$ and $m_\infty^{\otimes N}$, for any $\varepsilon\in(0,1)$,
\begin{align*}
\lefteqn{\mathbb E \po \mathcal C(\pi_{\bX},m_\infty)\pf}\\
  &\leqslant \alpha(1+\varepsilon)  \mathbb E \po \left\|\frac1N \sum_{i=1}^N \co \varphi(X_i) - \varphi(Y_i)\cf \right\|^2\pf + \alpha(1+\varepsilon^{-1}) \mathbb E \po \left\|\frac1N \sum_{i=1}^N \varphi(Y_i) - \int_{\R^d} \varphi \dd m_\infty \right\|^2\pf  \\
& \leqslant  \frac{\alpha (1+\varepsilon)  \|\na \varphi\|^2}{N }\mathcal W_2^2  \po m^N,m_\infty^{\otimes N}\pf  + \alpha (1+\varepsilon^{-1})   \frac{\Var_{m_\infty}(\varphi )}{N}\,,
\end{align*}
provided $\varphi$ is Lipschitz (and thus has finite variance with respect to $m_\infty$ when the latter satisfies a Poincaré inequality). This is exactly~\eqref{eq:cost-bound} with $\alpha_N = \alpha(1+\varepsilon^{-1})  \Var_{m_\infty}(\varphi ) $ for all $N\in\N$.

To check the other conditions of Corollary~\ref{cor:LSI}, we see that
\[D_m^2 F(\mu,x,x') = \na_x \varphi(x') \cdot \na^2 R\po \int_{\R^d} \varphi\dd \mu\pf \na_x \varphi(x) \,.\]
Since we already assumed that $\varphi$ is Lipschitz continuous, if we assume that $R$ is $\mathcal C^2$ with $\na^2 R$  bounded over the convex hull of $\{\varphi(x),\ x\in\R^d\}$ then we get Assumption~\ref{assu:Poincare}\ref{assu1iv}.

We may consider many conditions to ensure the conditional and local LSI. To keep this simple in this general framework, let us assume for instance that $\na R$ is also bounded over the convex hull of $\{\varphi(x),\ x\in\R^d\}$. In that case, since
\begin{align*}
\na_{x_1} NF(\mu_{\bx}) &= \na V(x_1) + \na R \po \frac1N\sum_{i=1}^N \varphi(x_i)\pf \cdot \na\varphi(x_1)\\
 \na_{x}\frac{\delta F}{\delta m}(\mu,x) &= \na V(x) +  \na R\po \int_{\R^d}\varphi\dd \mu\pf \cdot \na \varphi(x)\,,
\end{align*}
we get that $x_1\mapsto \exp(-NF(\mu_{\bx}))$ and $\exp(-\frac{\delta F}{\delta m}(\mu,\cdot))$ are Lipschitz perturbations of $e^{-V}$. Assuming that the latter satisfies a LSI, we get Assumptions~\ref{assu:Poincare}\ref{assu1v} (with a constant $\rho_N$) and \ref{assu:LSI}\ref{assu2ii}  thanks to \cite{AIDA1994448} (or  \cite{CattiauxGuillin} for explicit constants). 

As a conclusion, under the previous assumptions, we get a uniform-in-$N$ LSI for $m_*^N$ provided $\alpha$ is small enough.

\subsubsection{Application to a particular example}\label{sec:particularexemple}

In many models of pairwise interacting particles, the interaction is repulsive at short-range and attractive in long-range. This is the case for instance for swarming/flocking models   \cite{reynolds1987flocks}, cell adhesion \cite{flandoli2020macroscopic} or  molecular dynamics \cite{lelievre2016partial} with e.g. Lennard-Jones particles (not covered here as they are singular). Here we consider a simple example of the form
\[F(\mu) = \int_{\R^d} V \dd \mu + \frac12\int_{\R^{2d}} W(x-y) \mu(\dd x)\mu(\dd y)\,,\]
where  $W=W_1+W_2$, 
\[W_1(x-y) = L e^{-|x-y|^2}, \qquad W_2(x-y):= \alpha |x-y|^2\]
for some $L,\alpha>0$. When $L>\alpha$ we get repulsion at small distances and attraction otherwise.  In order to get a uniform-in-$N$ result, we can use \cite[Theorem 8]{guillin2022uniform}, but only if both $L$ and $\alpha$ are small enough (depending on $V$), or \cite[Theorem 1]{Songbo} for any $L>0$ if $\alpha=0$. By contrast, with Corollary~\ref{cor:LSI}, we can consider any $L>0$ and then some small $\alpha>0$ (the constraint on $\alpha$ does depend on $L$), as we detail below. This is interesting in cases such as mollified versions of Lennard-Jones-type particles, where the short-range repulsion is much stronger than the long-range attraction.

\begin{itemize}
\item \emph{Curvature condition.} By expanding $e^{2x\cdot y}$, as detailed in \cite[Exemple 1]{MonmarcheReygner},
\[F_1(\mu):= \frac{L}2\int_{\R^{2d}} e^{-|x-y|^2}  \mu(\dd x)\mu(\dd y) = \frac{L}2 \sum_{k=1}^\infty \po \int_{\R^d} w_k \dd \mu\pf^2 \]
with $w_k(x) =  e^{-|x|^2}\prod_{i=1}^d \frac{2^{k_i/2}x_i^{k_i}}{\sqrt{k_i!}} $. It is then clear that $F_1$ is flat-convex (the potential $W_1$ is said to be a Mercer kernel), and that~\eqref{eq:courbure-cost} holds with
\[\mathcal C(\mu,\nu) = \alpha |x-y|^2\,.\]
As shown in Section~\ref{sec:examplequimarche}  (with, in the present case, $\varphi(x)=x$), \eqref{eq:cost-bound} holds with a constant $\alpha_N$ and $\lambda'=\alpha(1+\varepsilon)$ for an arbitrary $\varepsilon\in(0,1)$.

\item \emph{Lipschitz condition.} The second flat derivative here is
\[\frac{\delta^2 F }{\delta^2 m}(\mu,x,x') = W(x-x')\,,\qquad D_m^2 F = 2L \po I_d-2|x-y|^{\otimes 2} \pf e^{-|x-y|^2} - 2\alpha I_d\]
from which $\|D_m^2 F\|_\infty \leqslant 2L(1+2e^{-1}) + 2\alpha =: M_{mm}^F$ (as $t\mapsto e^{-t}$ is maximal at $t=1$).

\item \emph{Local inequalities.} Assuming that $V=V_1+V_2$ where $V_1$ is bounded and $V_2$ is $\eta$-strongly convex for some $\eta>0$, for any measure $\mu$,
\[\frac{\delta F}{\delta m}(\mu,\cdot) = V + W\star \mu = V_1 + W_1\star \mu + V_2 + W_2\star\mu\,, \]
where $V_1+W_1\star\mu$ is bounded by $\|V_1\|_\infty + L$ and $V_2+W_2\star\mu$ is $\eta$-strongly convex. The Bakry-Emery curvature criterion and Holley-Stroock lemma (see \cite{BakryGentilLedoux}) give Assumption~\ref{assu:LSI}\ref{assu2ii} with $\rho = \eta e^{-\|V_1\|_\infty - L} $. Besides, for fixed $\bx_{\neq 1}$ and $N$, there exists $z,C\in\R$ such that 
\[N F(\mu_{\bx}) = V(x_1) + \alpha| x_1 - z|^2 + \frac{L}{2N} + \frac1N \sum_{j=2}^N W(x_1-x_j) +  C\,,\]
which is again, up to a constant, the sum of two potentials, one $\eta$-strongly potential, and one bounded by $\|V_1\|_\infty + L$. As a consequence, Assumption~\ref{assu:Poincare}\ref{assu1v} holds with $\rho_N =  \eta e^{-\|V_1\|_\infty - L}$.
\end{itemize}

As a conclusion, under the assumptions introduce above, Corollary~\ref{cor:LSI} applies and gives a uniform in $N$ LSI for $m_*^N \propto \exp(-U_N)$ as soon as 
\[4 \alpha < \eta e^{-\|V_1\|_\infty - L}\,.\]
Interested by the high temperature regime, we can introduce a parameter $\beta>0$ and consider the Gibbs measure $m_\beta^N \propto \exp(-\beta U_N)$, which simply consists in multiplying $F$ by $\beta F$. The constants $\alpha,L,\|V_1\|_\infty$ and $\eta$ are then all multiplied by $\beta$ and thus, assuming $4\alpha<\eta$, we obtain that $m_\beta^N $ satisfies a uniform in $N$ LSI as soon as 
\[ \beta  < \frac{\ln\eta - \ln (4 \alpha) }{\|V_1\|_\infty + L}  \,. \] 
Notice that, in order to get a uniform-in-$N$ LSI at some temperature in these settings, the condition $\eta>\alpha$ is necessary
since we are back to the case discussed in Section~\ref{sec:sharpness} when taking $L=0$, $V(x)= \frac{\eta}{2}|x|^2$.

\subsubsection{Some consequences of the defective log-Sobolev inequality}\label{sec:exemple-defective} 

Denote by $P_t^Nf(\bx) = \mathbb E_{\bx}\po f(\bX_t)\pf$ the semi-group associated with~\eqref{eq:Langevin} for nice $f:\R^{dN} \rightarrow \R$. This dynamics being reversible with respect to $m_*^N$, for an initial distribution $\bX_0 \sim \nu_0^N$ with a density $h_0 \in \mathcal C^1_b$ with respect to $m_*^N$, the  law  $\nu_t^N$ of $\bX_t$ has density $h_t = P_t^N h_0$ with respect to $m_*^N$, and an integration by parts yields
\[\partial_t \mathcal H \po \nu_t^N | m_*^N\pf = - \mathcal I \po \nu_t^N | m_*^N \pf \,.\]
Thanks to Theorem~\ref{thm:LSI}, under Assumptions~\ref{assu:Poincare} and \ref{assu:LSI} with $4\lambda' < \rho$, we get that 
\begin{equation*}
\partial_t \mathcal H \po \nu_t^N | m_*^N\pf \leqslant - c \mathcal H \po \nu_t^N | m_*^N\pf + C \alpha_N\,,
\end{equation*}
where $c,C>0$ are independent from $N$ large enough, and $\alpha_N$ is as~\eqref{eq:cost-bound}, and then
\begin{equation}
\label{eq:subdecay}
 \mathcal H \po \nu_t^N | m_*^N\pf \leqslant e^{-ct} \mathcal H \po \nu_0^N | m_*^N\pf + \frac{C}c \alpha_N\,.
\end{equation}
 As we saw in Section~\ref{sec:W2C}, in the general semi-convex where~\eqref{eq:courbure} holds for some $\lambda>0$, it is always possible to get $\alpha_N$ of order $N^{1-2/d}$, i.e. sub-linear. This makes~\eqref{eq:subdecay} interesting, because the relative entropy $\mathcal H \po \nu_t^N | m_*^N\pf$ is rather of order $N$ (if $\nu_0^N = m^{\otimes N}$ for some $m\in\mathcal P_2(\R^d)$ for instance). In particular, \eqref{eq:subdecay} with $\alpha_N =o(N)$ is reminiscent of  the approximate convergence of \cite[Theorem 3.4]{Pavliotis}. More precisely, with the vocabulary of \cite{Pavliotis}, we can reformulate the defective LSI as a regularized LSI:
 \[  \underset{\nu^N:\mathcal H \po \nu^N | m_*^N \pf > \varepsilon_N }{\inf}\frac{\mathcal I \po \nu^N | m_*^N \pf }{\mathcal H \po \nu^N | m_*^N \pf} \geqslant  \frac{c}{2}\,. \]
with $\varepsilon_N = 2C\alpha_N/c$.

Dividing \eqref{eq:subdecay} by $N$ and sending $N$ to infinity, we may follow \cite[Section 6.2]{guillin2022uniform} or \cite[Theorem 5.6 and Corollary 5.9]{Pavliotis} to get a similar convergence for the non-linear limit flow. In fact, following these references,  dividing the defective non-linear LSI by $N$, we may directly send $N$ to infinty in 
\[\frac{c}N \mathcal H \po m^{\otimes N} | m_*^N\pf \leqslant \frac{1}{N}\mathcal I \po m^{\otimes N} | m_*^N\pf +  \frac{C \alpha_N}N\]
since $\alpha_N =o(N)$, to get a non-linear LSI
\begin{equation}
\label{eq:nonlinLSI}
\forall m\in\mathcal P_2(\R^d)\,,\qquad  \mathcal F(m) - \inf \mathcal F \leqslant \frac{1}{2\bar\rho}\mathcal I \po m|\hat m \pf\,,\qquad \hat m \propto \exp\po - \frac{\delta F}{\delta}(m,\cdot)\pf\,,
\end{equation}
with $\bar \rho = c/2$, uniqueness of the critical point of the free energy as in  \cite[Theorem 3.6]{Pavliotis} and uniform in time propagation of chaos in $\mathcal W_2$ distance as in \cite[Theorem 3.7]{Pavliotis}.  That being said, under Assumptions~\ref{assu:Poincare} and \ref{assu:LSI} with $\lambda,\lambda'$ small enough with respect to $\rho$, such an inequality (and then uniqueness of the critical point) can be more easily obtained working directly with the non-linear flow, as in \cite[Section 2.3]{MonmarcheReygner}. 

This discussion leads to state a conjecture which is weaker than \cite[Conjecture 1]{Pavliotis}: assume that \eqref{eq:nonlinLSI} holds, and let $\bar\rho $ be the optimal constant to have this. Is is true then that there exists a sequence $(\rho_N,\delta_N)$ such that $m_*^N$ satisfies a defective LSI with constants $(\rho_N,\delta_N)$ for all $N\in\N^*$, with $\rho_N \rightarrow \bar\rho$ and $\delta_N = o(N)$ as $N\rightarrow \infty$ (Conjecture 1 of \cite{Pavliotis} being the same with $\delta_N=0$)?

A natural question also, in view of \cite{MonmarcheReygner} would be, in some cases where uniform-in-$N$ (defective) LSI fails (because the mean-field non-linear LSI fails)   to prove \emph{local} uniform-in-$N$ defective LSI, of the form
\[\forall \nu^N \in \mathcal A_N,\qquad \mathcal H\po  \nu^N|m_*^N\pf  \leqslant c \mathcal I(\nu^N|m_*^N) + \delta_N\,,\]
where $\delta_N = o(N)$ and $\mathcal A_N$ is a subset of $\mathcal P_2(\R^{dN})$, for instance the set of probability measures such that $\mathbb E\po \mathcal W_2^2(\mu_{\bX},\rho_*)\pf \leqslant \varepsilon$ where $\rho_*$ is a critical point of $\mathcal F$ and $\varepsilon>0$ is small enough, and then investigate possible consequences of such inequalities.

Besides, if the motivation is to minimize the non-linear free energy $\mathcal F$ with a particle system, starting from Theorem~\ref{thm:LSI}, we can obtain an approximate convergence in the spirit of \cite{ChenRenWang} or \cite[Section 5]{MonmarcheReygner} from~\eqref{eq:subdecay}, using entropy bounds as in Section~\ref{sec:entropyIneq}. 

These topics are out of the scope of this note. 

\section{Poincaré inequality}\label{sec:Poincare}

Here, the main difference with respect to the convex case is the following, whose role is to replace \cite[Lemma 4]{Songbo}.

\begin{lem}\label{lem:Hessienne}
Under Assumption~\ref{assu:Poincare}, for all $\bx,\bv\in (\R^d)^N$,
\[\sum_{i,j=1}^N v_i\cdot D_m^2 F(\mu_{\bx},x_i,x_j) v_j \geqslant - \lambda N \sum_{i=1}^N |v_i|^2\,.  \]
\end{lem}

\begin{proof}
Taking $t=1/2$ in \eqref{eq:courbure}, 
\begin{equation}
\label{eq:demo}
F\po \frac12 \mu_{\bx+h\bv} + \frac12 \mu_{\bx-h\bv}  \pf  \leqslant   \frac12 F\po \mu_{\bx+h\bv}\pf + \frac12 F \po \mu_{\bx-h\bv} \pf + \frac{\lambda}{8} \mathcal W_2^2 \po \mu_{\bx+h\bv},\mu_{\bx-h\bv}\pf  
\end{equation}
Coupling $\mu_{\bx+h\bv}$ and $\mu_{\bx-h\bv}$ by $(x_I+hv_I,x_{I}+hv_I)$ where $I$ is uniformly distributed over $\cco 1,N\ccf$ yields
\[\mathcal W_2^2 \po \mu_{\bx+h\bv},\mu_{\bx-h\bv}\pf  \leqslant \frac{4h^2 }{N} \sum_{i=1}^N |v_i|^2\,. \]
Besides,  using that $\frac{\delta^2 F}{\delta \mu^2}(\mu,z,z')$ is continuous in $\mu$ and $\mathcal C^2$ in $(z,z')$,
\begin{eqnarray*}
\lefteqn{ F\po \frac12 \mu_{\bx+h\bv} + \frac12 \mu_{\bx-h\bv}  \pf - \frac12 F\po \mu_{\bx+h\bv}\pf - \frac12 F \po \mu_{\bx-h\bv} \pf}\\
  &=&  - \frac12 \int_{(\R^d)^2} \frac{\delta^2 F}{\delta \mu^2 } (\mu_{\bx},z,z') \po  \mu_{\bx+h\bv} -  \mu_{\bx-h\bv}\pf ^{\otimes}(\dd z,\dd z') + \underset{h\rightarrow 0}o(h^2)\\
  &= &  -\frac12 \frac{h^2} {N^2} \sum_{i=1}^N \sum_{j=1}^N v_i \cdot \na^2_{x,y} \frac{\delta^2 F}{\delta \mu^2 } (\mu_{\bx},x_i,x_i) v_j+ \underset{h\rightarrow 0}o(h^2)\,.
  \end{eqnarray*}
  Conclusion follows by diving \eqref{eq:demo} by $h^2$ and sending $h$ to zero.
\end{proof}

\begin{proof}[Proof of Theorem~\ref{thm:Poincare}]
We follow Step 2 of the proof of \cite[Theorem 1]{Songbo}, based on the approach by \cite{guillin2022uniform} (itself an adaptation in the non-independent case of the usual proof of the tensorization property of Poincaré inequalities). Thanks to \cite[Proposition 4.8.3]{BakryGentilLedoux}, a Poincaré inequality for $m_*^N$ with some constant $\rho'>0$ is equivalent to
\begin{equation}\label{eq:preuvePoincareGamma2}
\forall f\in\mathcal C_b^2(\R^{dN}),\qquad \rho' \int_{\R^{dN}} |\na f|^2 \dd m_*^N \leqslant  \int_{\R^{dN}} \Gamma_2(f)\dd m_*^N\,,
\end{equation}
with  the $\Gamma_2$ Bakry-Emery operator (see \cite{BakryGentilLedoux}) which is here
\[\Gamma_2(f) = | \na^2 f|_{HS}^2 +  \na f \cdot \na U_N \na f\,, \] 
where $|\cdot|_{HS}$ stands for the Hilbert-Schmidt norm. Noticing that 
\[\na^2_{i,j} U_N(\bx) = \frac1N D_m^2 F(\mu_{\bx},x_i,x_j) + \na D_m F (\mu_{\bx},x_i) \1_{i=j} \]
and bounding $|\na^2 f|_{HS}^2 \geqslant \sum_{i=1}^N |\na_i^2 f|_{HS}^2$ gives
\begin{eqnarray*}
\int_{\R^{dN}} \Gamma_2(f) \dd m_*^N & \geqslant  & \sum_{i=1}^N \int_{\R^{dN}} \po |\na^2_i f|^2_{HS} + \na_i f \cdot D_m F(\mu_{\bx},x_i) \na_i f  \pf  \dd m_*^N\\
& &  + \frac1N \sum_{i,j=1}^N \int_{\R^{dN}} \na_i f \cdot D_m^2 F(\mu_{\bx},x_i,x_j) \na_j f    \dd m_*^N\,. 
\end{eqnarray*}
Denote by $A$ and $B$ respectively the term in the first and second line of the right hand side. Using  the uniform conditional Poincaré inequality (in its equivalent integrated $\Gamma_2$ form) and the Lipschitz condition of Assumption~\ref{assu:Poincare}$\ref{assu1iv}$,  it is established in the proof of \cite[Theorem 1]{Songbo} that
\[
A \geqslant \po \rho_N - \frac{M_{mm}^F}{N}\pf  \sum_{i=1}^N \int_{\R^d}  |\na_i f|^2     \dd m_*^N\,.
\]
Besides, Lemma~\ref{lem:Hessienne} gives
\[B \geqslant - \lambda \sum_{i=1}^N \int_{\R^d}  |\na_i f|^2     \dd m_*^N \,. \]
Hence, \eqref{eq:preuvePoincareGamma2} holds with $\rho' = \rho_N - \lambda - \frac{M_{mm}^F}{N}$, which concludes.
\end{proof}

\section{Log-Sobolev inequality}\label{sec:LSI}

In \cite{Songbo}, S. Wang use the defective log-Sobolev inequality established in his previous work \cite{ChenRenWang} with F. Chen and Z. Ren. Convexity is used in the proof, so we have to follow and adapt it. Notice that,  in \cite[Remark~5.2]{ChenRenWang}, the authors already mention that they expect their approach to work in the semi-convex case when the curvature is not too negative. Rather than the details for this extension, the main contribution of the present note is thus the introduction here of the condition~\eqref{eq:courbure-cost}, which allows to recover cases where $\alpha_N$ is uniformly bounded (hence a uniform log-Sobolev inequality in Theorem~\ref{thm:LSI}) which would not be the case simply with the semi-convexity condition~\eqref{eq:courbure} as discussed in Section~\ref{sec:W2C}.

In the rest of this section, Assumptions~\ref{assu:Poincare} and \ref{assu:LSI} are enforced. Before proceeding to the proof of Theorem~\ref{thm:LSI}, we start by establishing some preliminary bounds comparing the various entropies and free energies involved.

\subsection{Entropy inequalities}\label{sec:entropyIneq}

Dividing by $t$ and letting $t\rightarrow 0$ in~\eqref{eq:courbure-cost} gives the pent inequality
\begin{equation}
\label{eq:pent1}
\int_{\R^d} \frac{\delta F}{\delta m}(\mu)(\nu-\mu) \leqslant F(\nu)-F(\mu) + \mathcal C(\nu,\mu)
\end{equation}
or, simply rearranging the terms, 
\begin{equation}
\label{eq:pent2}
F(\mu)-F(\nu) \leqslant  \int_{\R^d} \frac{\delta F}{\delta m}(\mu)(\mu-\nu)  +  \mathcal C(\nu,\mu)\,.
\end{equation}
In particular, thanks to \eqref{eq:cost-bound}, taking the expectation of~\eqref{eq:pent2} with $\nu = \mu_{\bX}$ for $\bX\sim m^N$ and $\nu=m_\infty$ reads 
\begin{multline}
N \mathbb E \co \int_{\R^d} \frac{\delta F}{\delta m}(\mu_{\bX},y)(\mu_{\bX} - m_\infty)(\dd y)\cf  \geqslant N \mathbb E \co  F(\mu_{\bX})\cf - N F(m_\infty) \\
- \lambda' \mathcal W_2^2 \po m^N,m_\infty^{\otimes N}\pf - \alpha_N\,.\label{eq:entropyN}
\end{multline}
Proceeding similarly with~\eqref{eq:pent1} reads
\begin{multline}
 N \mathbb E \co F(\mu_{\bX}) \cf  - N F(m_\infty) \geqslant N  \mathbb E \co \int_{\R^d} \frac{\delta F}{\delta m}(m_\infty,y)(\mu_{\bX} - m_\infty)(\dd y)\cf  \\
- \lambda' \mathcal W_2^2 \po m^N,m_\infty^{\otimes N}\pf - \alpha_N\,.\label{eq:entropyN2}
\end{multline}

Besides, since $m_\infty$ is a minimizer of $\mathcal F$, it is a stationary solution of the associated Wasserstein gradient flow, which means that $m_\infty \propto \exp\po - \frac{\delta F}{\delta m }(m_\infty,\cdot)\pf$ (see \cite[Proposition 4.1]{Pavliotis}). In particular, Assumption~\ref{assu:LSI}$\ref{assu2ii}$ implies that it satisfies a LSI with constant $\rho$.

 For $N\in\N$ and $m^N \in\mathcal P_2(\R^{dN})$ we write the so-called $N$-particle free energy
 \[\mathcal F^N(m^N) = N\int_{\R^{dN}} F(\mu_{\bx}) m^N(\dd \bx) + H(m^N)\,.\]
 Since  
 \[\mathcal F^N(m^N) - \mathcal F^N(m_*^N)  = \mathcal H\po m^N | m_*^N\pf\,,\]
 in particular $\mathcal F^N$ is lower bounded and reaches its infimum at $m_*^N$.

Adapting the proof of  \cite[Lemma 5.2]{ChenRenWang} to our case thanks to~\eqref{eq:entropyN2},
\begin{eqnarray}
\mathcal F^N(m^N) - N \mathcal F(m_\infty)  & = & N \mathbb E \po  F(\mu_{\bX}) - F(m_\infty) \pf + H(m^N) - N H(m_\infty) \nonumber  \\
& \geqslant & N  \mathbb E \co \int_{\R^d} \frac{\delta F}{\delta m}(m_\infty,y)(\mu_{\bX} - m_\infty)(\dd y)\cf \nonumber \\
& & \ +\  H(m^N) - N H(m_\infty)  - \lambda' \mathcal W_2^2 \po m^N,m_\infty^{\otimes N}\pf - \alpha_N \nonumber  \\
  & = & \mathcal H\po m^N |m_\infty^{\otimes N}\pf  - \lambda' \mathcal W_2^2 \po m^N,m_\infty^{\otimes N}\pf - \alpha_N \label{eq:FNetH} \,.
\end{eqnarray}

Since $m_\infty^{\otimes N}$  satisfies a LSI with constant $\rho$, it also satisfies the Talagrand inequality
\begin{equation}
\label{eq:Talagrand}
\mathcal W_2^2\po m^N,m_\infty^{\otimes N}\pf \leqslant \frac{2}{\rho} \mathcal H\po m^N|m_\infty^{\otimes N}\pf \,,
\end{equation}
which together with~\eqref{eq:FNetH} gives
\begin{equation}
\mathcal F^N(m^N) - N \mathcal F(m_\infty)  \geqslant  \po 1 - \frac{2\lambda'}{\rho}\pf  \mathcal H\po m^N |m_\infty^{\otimes N}\pf   - \alpha_N 
\label{eq:FNetH3}
\end{equation}
When $2\lambda'\leqslant \rho$, applying this with $m^N = m_*^N$ gives
\[\mathcal F^N(m_*^N) - N \mathcal F(m_\infty) \geqslant - \alpha_N \,,\]
and thus
\begin{equation}
\label{eq:FNetH2}
\mathcal H(m^N|m_*^N) = \mathcal F^N(m^N) - \mathcal F^N(m_*^N) \leqslant \mathcal F^N(m^N) - N \mathcal F(m_\infty) + \alpha_N\,.
\end{equation}

\subsection{Defective log-Sobolev inequality}

Under Assumption~\ref{assu:Poincare}, in the proof of \cite[Theorem 2.6]{ChenRenWang}, the following is established, without requiring yet  that $F$ is convex: for all $\varepsilon\in(0,1)$,
\begin{multline}
\mathcal I \po m^N|m_*^N\pf \geqslant - \po \varepsilon^{-1}-1\pf \Delta_1  
+ 4\rho (1-\varepsilon) \Big[  N \int_{\R^{dN}} \int_{\R^d} \frac{\delta F}{\delta m}(\mu_{\bx},y)(\mu_{\bx} - m_\infty)(\dd y)m^N(\dd \bx)\\  + H(m^N) - N H(m_\infty) - \Delta_2\Big] \label{eq:LSI1}
\end{multline}
(this is the combination of Equations (5.1), (5.5) and (5.6) of \cite{ChenRenWang}) where $\Delta_1$ and $\Delta_2$ are error terms (related to two change of measures, the first one performed to isolate one particle to be independent from the other and then the second one to go back to the true joint distribution) satisfying  
\begin{eqnarray*}
|\Delta_1| & \leqslant & 6 (M_{mm}^F)^2  \po \frac{1}{N}\mathcal W_2^2\po m^N,m_\infty^{\otimes N}\pf  + \frac{d}{2\rho}\pf  \\
|\Delta_2| & \leqslant & M_{mm}^F \po \frac{4}{N}\mathcal W_2^2\po m^N,m_\infty^{\otimes N}\pf  + \frac{5d}{2\rho}\pf 
\end{eqnarray*}
(this is (5.12) and (5.13) of \cite{ChenRenWang} with the bound on the variance of $m_\infty$ implied by the Poincaré inequality which follows from the uniform LSI of Assumption~\ref{assu:LSI}$\ref{assu2ii}$). 
Inserting \eqref{eq:entropyN} in~\eqref{eq:LSI1} yields
\begin{multline}
\mathcal I \po m^N|m_*^N\pf \geqslant - \po \varepsilon^{-1}-1\pf \Delta_1  \\
+ 4\rho (1-\varepsilon) \co  \mathcal F^N(m^N) - N \mathcal F(m_\infty) - \Delta_2 - \lambda' \mathcal W_2^2 \po m^N,m_\infty^{\otimes N}\pf - \alpha_N\cf\\
\geqslant  4\rho (1-\varepsilon) \co  \mathcal F^N(m^N) - N \mathcal F(m_\infty) - \tilde \lambda_N \mathcal W_2^2 \po m^N,m_\infty^{\otimes N}\pf - \tilde \alpha_N \cf \,,\label{eq:fin}
\end{multline}
with $\tilde \lambda_N$ given by \eqref{eq:lambdatilde} and
\begin{align*}
\tilde \alpha_N &=  \alpha_N  + \frac{M_{mm}^F d}{\rho } \po \frac52 +  \frac{3M_{mm}^F (\varepsilon^{-1}-1)}{4\rho(1-\varepsilon)} \pf \,.  
\end{align*}
Recall the definition of $\beta_N$ in~\eqref{eq:betaN}.  The condition  $N>N_0$ implies that $4\tilde \lambda_N < \rho$, itself  equivalent to $\beta_N \in(0,1)$.  For $N>N_0$, applying the Talagrand inequality~\eqref{eq:Talagrand},  then~\eqref{eq:FNetH3} and finally~\eqref{eq:FNetH2}
\begin{eqnarray*}
  \mathcal F^N(m^N) - N \mathcal F(m_\infty) - \tilde \lambda_N \mathcal W_2^2 \po m^N,m_\infty^{\otimes N}\pf  
 & \geqslant &  \mathcal F^N(m^N) - N \mathcal F(m_\infty) - \frac{2\tilde \lambda_N}{\rho}\mathcal H\po m^N|m_\infty^{\otimes N}\pf \\
  & \geqslant & \po 1 - \beta_N \pf  \po  \mathcal F^N(m^N) - N \mathcal F(m_\infty) \pf - \beta_N \alpha_N \\
  & \geqslant & \po 1 - \beta_N \pf   \mathcal H(m^N|m_*^N) - \alpha_N \,.
\end{eqnarray*}
Inserting this in \eqref{eq:fin}, we obtain the following defective LSI: for $N>N_0$ and any $\varepsilon\in(0,1)$, for all smooth  $m^N \in \mathcal P_2(\R^{dN})$, 
\begin{equation}
\label{eq:defectiveLSI}
\mathcal I \po m^N|m_*^N\pf \geqslant  4\rho (1-\varepsilon) (1-\beta_N) \mathcal H(m^N|m_*^N) + \delta_N
\end{equation}
with $\delta_N = 4\rho(1-\varepsilon)( \alpha_N + \tilde \alpha_N)$.

\subsection*{Acknowledgments}

This research is supported by the projects SWIDIMS (ANR-20-CE40-0022) and CONVIVIALITY (ANR-23-CE40-0003) of the French National Research Agency. P.M. thanks Katharina Schuh for fruitful discussions on the work \cite{Songbo2}, related to the present work.

\bibliographystyle{plain} 
\bibliography{biblio}

\end{document}